\documentclass{amsart}
\usepackage{amsmath,amsfonts,amssymb,amscd,verbatim,multicol}
\usepackage{fullpage}

\usepackage{graphicx}
\usepackage{stmaryrd}
\usepackage{mathtools}
\usepackage{dsfont}
\usepackage{xcolor}
\usepackage{enumitem}

\numberwithin{equation}{section}
\newtheorem{theorem}[equation]{Theorem} 

\newtheorem{corollary}[equation]{Corollary} 
\newtheorem{lemma}[equation]{Lemma}
\newtheorem{proposition}[equation]{Proposition}

\newtheorem{definition}[equation]{Definition} 
 
\newtheorem{example}[equation]{Example}
\newtheorem{hypothesis}[equation]{Hypothesis} 

\newtheorem{remark}[equation]{Remark}

\newtheorem*{theorem*}{Theorem}

\DeclareMathOperator\Aut{Aut}

\DeclareMathOperator\Ext{Ext}

\DeclareMathOperator\gldim{gldim}

\DeclareMathOperator\id{id}

\DeclareMathOperator\lcm{lcm}

\DeclareMathOperator\ord{ord}
\DeclareMathOperator\orb{orb}

\DeclareMathOperator\pdim{pdim}

\DeclareMathOperator\stb{stb}

\newcommand\NN{\mathbb N}

\newcommand\ZZ{\mathbb Z}

\newcommand\fsl{\mathfrak{sl}}

\newcommand\bp{\mathbf p}
\newcommand\bq{\mathbf q}

\renewcommand{\int}{\mathrm{int}}
\newcommand\inv{^{-1}}

\newcommand\iso{\cong}
\newcommand\niso{\not\iso}
\newcommand\kk{\Bbbk}

\newcommand\tensor{\otimes}

\newcommand\poly{\kk[h]}
\newcommand\lnt{\kk[h^{\pm 1}]}

\newcommand\restr[2]{{
  \left.\kern-\nulldelimiterspace 
  #1 
  \vphantom{\big|} 
  \right|_{#2} 
  }}

\newcommand{\grp}[1]{{\langle {#1} \rangle}}

\renewcommand{\to}{\ensuremath{\longrightarrow}}

\renewcommand\mod{~\mathrm{mod}~}

\begin{document}

\title{Fixed rings of quantum generalized Weyl algebras}
\author[Gaddis]{Jason Gaddis}
\author[Ho]{Phuong Ho}
\address{Miami University, Department of Mathematics, Oxford, Ohio 45056} 
\email{gaddisj@miamioh.edu,hopu@miamioh.edu}



\begin{abstract}
Generalized Weyl Algebras (GWAs) appear in diverse areas of mathematics including mathematical physics, noncommutative algebra, and representation theory. We study the invariants of quantum GWAs under finite automorphisms. We extend a theorem of Jordan and Wells and apply it to determine the fixed ring of quantum GWAs under diagonal automorphisms. We further study properties of the fixed rings including global dimension, the Calabi-Yau property, rigidity, and simplicity. 
\end{abstract}

\maketitle

\section{Introduction}

Throughout, $\kk$ is a field and all algebras are associative $\kk$-algebras. 

The Shephard-Todd-Chevalley (STC) Theorem \cite{Chev,ShTo} gives conditions for the fixed ring of a polynomial ring by a finite group of linear automorphisms to again be a polynomial ring. More recently, there has been significant interest in studying STC-like theorems in noncommutative algebra, in particular whether the fixed ring of an ($\NN$-graded) Artin-Schelter regular algebra again has this property \cite{KKZ1}.

To consider this problem outside of the $\NN$-graded setting, one could ask whether the fixed ring of a (twisted) Calabi-Yau algebra is again (twisted) Calabi-Yau.
Algebras satisfying this property have attracted much interest of late \cite{ginz,L,RRZ,RR2}. 
Since polynomial rings are (trivially) Calabi-Yau, this is indeed a reasonable generalization.

An important family of $\ZZ$-graded (twisted) Calabi-Yau algebras are the generalized Weyl algebras (GWAs) and so they serve as a good test case of the STC question in this setting.
Kirkman and Kuzmanovich \cite{KK} have proposed a version of the STC Theorem for GWAs, essentially asking when the fixed ring of a GWA again has GWA structure. We propose a strengthening of this: to determine when the fixed ring of a Calabi-Yau GWA is again a Calabi-Yau GWA.

Generalized Weyl algebras were named by Bavula \cite{B1}. They have been studied extensively by many authors prior to and post Bavula's definition. Notably, the interested reader is directed to the work of Hodges \cite{H2,H1}, Jordan \cite{Jkrull}, Joseph \cite{joseph}, Rosenberg \cite{rosenberg}, Smith \cite{smith3}, and Stafford \cite{St2}.
There are several important families of algebras that may be constructed as GWAs. In the classical case, this includes the Weyl algebras and primitive quotients of $U(\fsl_2)$.
We will primarily be concerned with a subclass known as \emph{quantum GWAs}, which includes quantum planes, quantum Weyl algebras, and primitive quotients of $U_q(\fsl_2)$.

The proposal of Kirkman and Kuzmanovich \cite{KK} mentioned above has its basis in the work of Jordan and Wells \cite{JW} and their study of fixed rings of GWAs by automorphisms that fix the base ring.
Won and the first-named author showed that it is possible to diagonalize any filtered automorphism of a classical GWA with quadratic defining polynomial so that the Jordan and Wells result may be applied \cite{GW1}. The methods and results in this paper vary wildly from these aforementioned works.

\subsection*{Results}
First, in Section \ref{sec.autos}, we classify finite subgroups of the automorphism group of a quantum GWA (Proposition \ref{prop.subgrp}). 
In Section \ref{sec.JW} we prove a generalization of the Jordan and Wells theorem (Theorem \ref{thm.gJW}). This allows us to show that, under certain mild conditions, the fixed ring of a quantum GWA by a diagonal automorphism is again a quantum GWA (Corollary \ref{cor.fixGWA1}). Here an automorphism is diagonal if it acts by scalar multiplication on a canonical generating set of these quantum GWAs. 

Finally, in Section \ref{sec.props}, we study properties of the fixed ring.
In \cite[Corollary 2.12]{GW1} it was shown that the global dimension of the fixed ring of a classical GWA could be computed easily from the global dimension of the original GWA and the order of the filtered automorphism group. For quantum GWAs, the global dimension depends on multiple factors and there are many cases that we must consider. One surprising case involves quantum GWAs with infinite global dimension whose fixed rings have finite global dimension. See Theorem \ref{thm.gldim} for a full explanation of global dimension in fixed rings of quantum GWAs.

Though this paper covers an analysis of fixed rings of quantum GWAs in the most significant cases, there is still some work to be done for a full analysis. In particular, we do not fully consider the case where the defining polynomial is \emph{symmetric} (see Definition \ref{def.sym}).

\section{GWAs and their automorphisms}
\label{sec.autos}

\begin{definition}
Let $D$ be a commutative algebra, $\sigma \in \Aut(D)$, and $a \in D$, $a \neq 0$. 
The \emph{(degree one) generalized Weyl algebra} (GWA) $D(\sigma,a)$ is generated over $D$ by $x$ and $y$ subject to the relations
\begin{align*}
&xd=\sigma(d)x, \quad yx=\sigma\inv(d)y, \quad\text{for all $d \in D$}, \\
&yx=a, \quad xy=\sigma(a).
\end{align*}
A GWA $R=D(\sigma,a)$ is \emph{classical} if $D=\poly$ and $\sigma(h)=h-\alpha$, $\alpha \in \kk^\times$. A GWA $R=D(\sigma,a)$ is \emph{quantum} if $D=\poly$ or $\lnt$ and $\sigma(h)=qh$, $q \in \kk\backslash\{0,1\}$.
\end{definition}
\noindent Some authors consider the case $D=\lnt$ above with $q$ a nonroot of unity to also be classical. We do not make that definition here to avoid confusion in our discussion. Additionally, one may define higher degree GWAs, see \cite{B1}, but we do not consider those here.

\begin{example}
The \emph{quantum planes} are quantum GWAs over $\kk[h]$ with $a=h$. The \emph{quantum Weyl algebras} are quantum GWAs over $\poly$ with $a=h-1$. Meanwhile, the minimal primitive factors of $U_q(\fsl_2)$ are certain quantum GWAs over $\lnt$ with $a$ quadratic. 
\end{example}

Bavula and Jordan considered automorphisms and isomorphisms of quantum GWAs in the case $D=\lnt$ \cite{BJ}. Richard and Solotar studied the isomorphism problem in the case $D=\poly$ \cite{RS2}. Both of these works required $q$ to be a nonroot of unity. Su\'{a}rez-Alvarez and Vivas solved the isomorphism problem in both cases, $D=\poly$ and $D=\lnt$, for $q\neq 1$. Additionally, they fully classified the automorphism group in the case $D=\poly$, again for $q \neq 1$, and gave a near complete classification when $D=\lnt$. In this section we recall their results and classify finite subgroups of automorphisms of most quantum GWAs.

By \cite[Theorem A]{SAV}, two quantum GWAs $R=D(\sigma,a)$ and $R'=D(\sigma',a')$ with parameters $q,q'$, respectively, are isomorphic if and only if $q'=q^{\pm 1}$ and there exists $\alpha \in D^\times$, $\beta \in \kk^\times$, and $\epsilon \in \{\pm 1\}$ such that $a(h) = \alpha a'(\beta h^\epsilon)$. Clearly $\epsilon=1$ if $D=\kk[h]$. This result allows us to always assume $a(h)$ is monic and of positive degree. 

Let $R=D(\sigma,a)$ be a quantum GWA with $a$ not a unit. 
Write $a = \Sigma_{i \in I} a_i h^i$ where $a_i \in \kk$ and $I \subset \ZZ$ is a finite set consisting of those $i$ such that $a_i \neq 0$. Note if $D=\kk[h]$ then in fact $I \subset \NN$.
Let $g = \gcd\{i-j: a_i a_j \neq 0\}$. If $a$ is a monomial, which occurs only when $D=\poly$, then let $C_g=\kk^\times$ and otherwise let $C_g$ be the subgroup of $\kk^\times$ consisting of $g$th roots of unity.
Fix $i_0 \in I$. If $(\gamma,\mu) \in C_g \times D^\times$, then there is an automorphism $\eta_{\gamma,\mu}$ of $R$ such that 
\begin{align}
\label{eq.eta}
\eta_{\gamma,\mu}(h) = \gamma h,\quad 
\eta_{\gamma,\mu}(y) = y\mu, \quad
\eta(x) = \mu^{-1}\gamma^{i_0} x.
\end{align}
The choice of notation $y\mu$ is intentional. In the case $D=\lnt$ we may have $\mu=h^k$, $k \in \ZZ$, and $y\mu \neq \mu y$ in general. However, we will primarily focus on the case that $\mu \in \kk^\times$.

\begin{remark}
\label{rmk.indep}
It is not difficult to see that $\eta_{\gamma,\mu}$ is independent of choice of $i_0 \in I$.
Suppose $j_0 \in I$ is a distinct choice.
Then $g \mid i_0-j_0$ and hence $\gamma^{i_0-j_0}=1$, so $\gamma^{i_0}=\gamma^{j_0}$.
When $D=\poly$ we often take $i_0=\deg_h(a)$.
\end{remark}

Let $G$ be the group generated by the $\eta_{\gamma,\mu}$.
When $q=-1$, there exists $\Omega \in \Aut(R)$ given by
\[ 
\Omega (h) = -h, \quad
\Omega (y)= x, \quad
\Omega (x) = y.\]
Note that $\Omega^2 = \id_R$.
If $R=\kk[h](\sigma,a)$ is a quantum GWA, 
then by \cite[Theorem B]{SAV}, $\Aut(R)=G$ unless $q=-1$, in which case $\Aut(R)=G \rtimes \{\Omega\}$.

The case $D=\lnt$ has added complications.

\begin{definition}
\label{def.sym}
A polynomial $a(h) \in \lnt$ is \emph{symmetric} if there exists $l \in \ZZ$ and $\delta,\lambda \in \kk^\times$ such that $a(h)=\delta h^l a(\lambda h\inv)$.
\end{definition}

If $a$ is not symmetric, then by \cite[Theorem C]{SAV}$, \Aut(R)=G$ unless $q=-1$, in which case $\Aut(R)=G \rtimes \{\Omega\}$, just as above.
If $a$ is symmetric, there exists $\Psi \in \Aut(R)$ given by
\[ 
\Psi (h) = q\inv \lambda h\inv, \quad
\Psi (y)= x, \quad
\Psi (x) = \delta\inv q^{-l} yh^{-l}.\]
In this case there are additional automorphisms $\phi \notin G$ that act diagonally on $\{h,x,y\}$.
Denote the set of automorphisms such that $\phi(h)$ is a scalar multiple of $h$ by $K$. Then there is a short exact sequence
\[ 0 \rightarrow K \rightarrow \Aut(A) \rightarrow \ZZ/2\ZZ \rightarrow 0.\]
We do not attempt to understand these automorphisms as they do not appear to be of finite order. Instead, we focus on the automorphisms $\eta_{\gamma,\mu}$. 

\begin{proposition}
\label{prop.autprops}
Let $R$ be a quantum GWA and let $\eta_{\gamma,\mu},\eta_{\gamma',\mu'} \in \Aut(R)$ for appropriate choices of parameters $\gamma,\gamma' \in \kk^\times$ and $\mu,\mu' \in D^\times$.
We have the following composition rules:
\begin{enumerate}
\item $\eta_{\gamma',\mu'}\circ\eta_{\gamma,\mu}= \eta_{\gamma' \gamma,\gamma^{\deg \mu}\mu'\mu}$,
\item $\eta_{\gamma,\mu}\inv =\eta_{\gamma^{-1},(\gamma^{\deg\mu}\mu)^{-1}}$, and
\item $\Omega\circ\eta_{\gamma,\mu} = \eta_{\gamma,\gamma^{i_0} \mu^{-1}}\circ\Omega$ in the case $q=-1$.
\end{enumerate}
\end{proposition}
\begin{proof}
We will check each claim by verifying that the relation holds for the generators. Note that $\mu$ is some monomial in $\lnt$. Hence, for $\gamma \in \kk^\times$ and $\mu,\mu' \in D^\times$, we have $\eta_{\gamma,\mu}(\mu') = \gamma^{\deg \mu'} \mu'$ and $\Omega(\mu)=(-1)^{\deg\mu}$. 

For (1), 
\begin{align*}
(\eta_{\gamma',\mu'}\circ\eta_{\gamma,\mu})(h)    
    &= \eta_{\gamma',\mu'}(\gamma h)
    = (\gamma'\gamma)h
    = \eta_{\gamma\gamma',\gamma^{\deg \mu}\mu\mu'}(h)\\
(\eta_{\gamma',\mu'}\circ\eta_{\gamma,\mu})(y) 
    &= \eta_{\gamma,\mu}(y\mu) 
    = (y\mu') (\gamma^{\deg \mu} \mu) 
    = \eta_{\gamma\gamma',\gamma^{\deg \mu}\mu\mu'}(y) \\
(\eta_{\gamma',\mu'}\circ\eta_{\gamma,\mu})(x) 
    &= \eta_{\gamma',\mu'}(\gamma^{i_0} \mu^{-1} x) 
    = (\gamma^{i_0} (\gamma^{\deg \mu} \mu)^{-1})( (\gamma')^{i_0} (\mu')^{-1} x) \\
    &= (\gamma\gamma')^{i_0} (\gamma^{\deg \mu}\mu\mu')^{-1} x
    = \eta_{\gamma\gamma',\gamma^{\deg \mu}\mu\mu'}(x). 
\end{align*}
Now (2) follows directly from (1) by observing that $\eta_{1,1}$ is the identity map. For (3), 
\begin{align*}
(\Omega\circ\eta_{\gamma,\mu})(h)
    &= \Omega(\gamma h)=-\gamma h 
    = \eta_{\gamma,\gamma^{i_0} \mu^{-1}}(-h)
    = (\eta_{\gamma,\gamma^{i_0} \mu^{-1}}\circ\Omega)(h)\\
(\Omega\circ\eta_{\gamma,\mu})(y) 
    &= \Omega(y \mu )
    = (-1)^{\deg\mu} x \mu 
    = \mu x
    = \gamma^{i_0} (\gamma^{i_0} \mu\inv)\inv x
    = \eta_{\gamma, \gamma^{i_0} \mu^{-1}}(x) 
    = (\eta_{\gamma,\gamma^{i_0}\mu^{-1}}\circ\Omega)(y)\\
(\Omega\circ\eta_{\gamma,\mu})(x) 
    &= \Omega(\gamma^{i_0} \mu^{-1}x)
    = (-1)^{\deg\mu} \gamma^{i_0} \mu^{-1} y 
    = y (\gamma^{i_0}\mu\inv)
    = \eta_{\gamma, \gamma^{i_0} \mu^{-1}}(y)
    = (\eta_{\gamma,\gamma^{i_0} \mu^{-1}}\circ\Omega)(x).\qedhere
\end{align*}
\end{proof}

For a root of unity $\gamma $ we denote by $\ord(\gamma)$ its order in the multiplicative group $\kk^\times$. We use this notation also for the order of an automorphism. When $\gamma,\mu \in \kk^\times$, $\ord(\eta_{\gamma,\mu})=\lcm(\ord(\gamma),\ord(\mu))$. The following should be compared to \cite[Theorem 3.1]{GW1}. 

\begin{proposition}
\label{prop.subgrp}
Let $R=D(\sigma,a)$ be a quantum GWA
with $a$ not symmetric. If $H$ is a finite subgroup of $G=\Aut(R)$, then one of the following holds:
\begin{enumerate}
    \item $H=\grp{\eta_{\gamma,\mu}}$ for some $\gamma, \mu \in \kk^\times$,
    \item $H=\grp{\Omega \circ \eta_{\gamma,\mu}}$ for some $\gamma \in \kk^\times$ and $\mu \in D^\times$,
    or
    \item $H=\grp{\Omega \circ \eta_{\gamma,\mu},\eta_{\gamma',\mu'}}$ for some $\gamma,\gamma',\mu' \in \kk^\times$ and $\mu \in D^\times$.
\end{enumerate}
If $q\neq -1$, then only case (1) holds.
\end{proposition}
\begin{proof}
Let $K$ be the subgroup of $H$ consisting of finite order automorphisms of the form $\eta_{\gamma,\mu}$. 
It is clear that $\ord(\eta_{\gamma,\mu}) < \infty$ if and only if $\mu \in \kk^\times$ and $\ord(\gamma),\ord(\mu) < \infty$.
Since automorphisms of this form commute, then $K$ is generated by a single $\eta_{a,b}$ with $a,b \in \kk^\times$. Hence, if $H=K$ then $H \iso C_k$ where $k=\lcm(\ord(a),\ord(b))$.

By Proposition \ref{prop.autprops}, for $\gamma,\gamma' \in \kk^\times$ and $\mu,\mu' \in D^\times$, we have
\begin{align}
\label{eq.omega-eta1}
\Omega \circ \eta_{\gamma',\mu'} \circ \Omega \circ \eta_{\gamma,\mu}
= \eta_{\gamma',(\gamma')^{i_0}(\mu')\inv} \circ \Omega \circ \Omega \circ \eta_{\gamma,\mu}
= \eta_{\gamma'\gamma,(\gamma')^{i_0}\gamma^{\deg\mu}(\mu')\inv\mu} \in K.
\end{align}
Observe that, since this composition has finite order, then $(\mu')\inv\mu \in \kk^\times$.

Suppose $K=\id$.
Applying \eqref{eq.omega-eta1} to $\gamma=\gamma'$ and $\mu=\mu'$ we have 
$(\Omega \circ \eta_{\gamma,\mu})^2 = \eta_{\gamma^2,\gamma^{i_0+\deg\mu}}=\id$,
whence $\gamma^2=\gamma^{i_0+\deg\mu}=1$, so $\gamma=\pm 1$ and $H \iso C_2$ or $H \iso C_4$.

Now suppose $H \neq K$, so $H$ contains some $\Omega \circ \eta_{\gamma,\mu}$. If $\Omega \circ \eta_{\gamma',\mu'} \in H$, then by \eqref{eq.omega-eta1}, we have $(\Omega \circ \eta_{\gamma',\mu'}) \circ (\Omega \circ \eta_{\gamma,\mu}) = \eta_{a,b}^\ell \in K$ for some $\ell \in \ZZ$.
Thus, $\Omega \circ \eta_{\gamma',\mu'} = (\Omega \circ \eta_{\gamma,\mu})\inv \circ \eta_{a,b}^\ell \in \grp{\eta_{a,b},\Omega \circ \eta_{\gamma,\mu}}$.
\end{proof}

\begin{remark}
Of course, the subgroups given in Proposition \ref{prop.subgrp} are also finite subgroups in the case that $a$ is symmetric. We conjecture that these constitute \emph{all} of the finite subgroups even when $a$ is symmetric.
\end{remark}

\section{An extension of a theorem of Jordan and Wells}
\label{sec.JW}

Let $R=D(\sigma,a)$ be a GWA. Set $\mu \in \kk^\times$ with $m=\ord(\mu)$ and define $\phi \in \Aut(R)$ by $\phi(d)=d$ for all $d\in D$, $\phi(x) = \mu\inv x$, and $\phi(y) = \mu y$.
By Jordan and Wells \cite[Theorem 2.6]{JW}, $R^\grp{\phi}$ = $D(\sigma^m, A)$ with $A=\prod^{m-1}_{i=0} \sigma^i(a)$. We want to study a more general version of their theorem in which we allow that $\phi|_D$ is not necessarily the identity.
First, we will need some further background on GWAs.

Setting $\deg(D)=0$, $\deg(x)=1$, and $\deg(y)=-1$ defines a $\ZZ$-grading on the GWA $R=D(\sigma,a)$. We denote by $R_k$ the $\kk$-vector space of degree $k$ homogeneous elements in $R$. 
Using the GWA relations, one has $R_0=D$, $R_k=x^kD$ for $k>0$, and $R_k=y^{-k}D$ for $k<0$.
Suppose $\phi \in \Aut(R)$ respects the $\ZZ$-grading on $R$. 
Then $\phi$ is a linear transformation on each $R_k$ and hence to determine $R^\grp{\phi}$ it suffices to determine the fixed space of each graded component of $R$. That is,
$(R^\grp{\phi})_k=(R_k)^\grp{\phi}$.
Moreover, the following identities in $R$ can be established by an easy induction argument,
\begin{align}
\label{eq.yx}
\displaystyle y^m x^m &= \sigma^{-(m-1)}(a) \cdots \sigma^{-2}(a) \sigma^{-1}(a)a=\prod^{m-1}_{i=0} \sigma^{-i}(a), \\
\label{eq.xy}
\displaystyle x^m y^m &= \sigma^{m}(a) \sigma^{m-1}(a) \cdots \sigma^2(a)\sigma(a)=\prod^{m}_{i=1} \sigma^{i}(a).
\end{align}

\begin{theorem}
\label{thm.gJW}
Let $D$ be an integral domain, $R=D(\sigma,a)$ a GWA, and $\mu \in \kk^\times$ with $m=\ord(\mu)<\infty$.
Suppose that $\phi \in \Aut(R)$ satisfies:
\begin{enumerate}
\item $\restr{\phi}{D}$ is an automorphism of $D$ with $n=\ord(\restr{\phi}{D})<\infty$;
\item $\phi(x) = \mu^{-1}x$ and $\phi(y) = \mu y$.
\end{enumerate}
If $\gcd(m,n)=1$, then $R^\grp{\phi}$ = $D^\grp{\phi}(\sigma^m, A)$ with $A=\prod^{m-1}_{i=0} \sigma^{-i}(a)$.
\end{theorem}
\begin{proof}
Assume $\gcd(n,m)=1$. We want to prove that the fixed elements form a GWA of the prescribed form. We have $\phi(y^m) = \mu^m y^m = y^m$ and $\phi(x^m) = \mu^{-m} x =  x^m$.
Set $Y=y^m$ and $X=x^m$.
It follows directly that $Xd=\sigma^m(d)X$ and $Yd=\sigma^{-m}(d)Y$ for all $d \in D^\grp{\phi}$. By \eqref{eq.yx} and \eqref{eq.xy}, $YX=A$ and $XY=\sigma^m (A)$. Hence, $X,Y$, and $D^\grp{\phi}$ generate the GWA $D^\grp{\phi}(\sigma^m, A)$. It remains to show the these elements generate the fixed ring. 

It is clear that $R_0^\grp{\phi}=D^\grp{\phi}$. Let $k \in \NN$, $k>0$, and $p\in D$. 
We denote by $\orb(p)$ and $\stb(p)$ the orbit and stabilizer, respectively, of $p$ under the action of $\restr{\phi}{D}$.
Assume $x^k p$ is fixed by $\phi$. Then $x^k p = \phi(x^k p) = \mu^{-k}x^k \phi(p)$. Because $R$ is a domain, then $\phi(p) = \mu^{k} p$. We then have $\phi^m(p)=(\mu^k)^m p =(\mu^m)^k p  = p$. Hence, $\orb(p)$ divides $m$.
According to Orbit-Stabilizer theorem, $\orb(p)\cdot\stb(p) = \ord\left( \restr{\phi}{D} \right)=n$. Then $\orb(p)$ divides $n$ and this contradicts $\gcd(m,n) = 1$ unless $m=n=1$, in which case $\phi=\id_R$.
It follows that $R_k^\grp{\phi}=\{0\}$ if $k$ does not divide $m$ and otherwise $R_{mk}=x^{mk}D^\grp{\phi}$. The proof is similar for $k<0$.
\end{proof}

\begin{example}
Let $D=\kk[h_1,h_2]$ and $R=D(\sigma,a)$ with $\sigma(h_1)=h_2$, $\sigma(h_2)=h_1$, and $a=h_1h_2$.
Define an automorphism $\eta:R \to R$ by $\eta(h_i)=-h_i$, $\eta(y)=\omega y$, $\eta(x)=\omega\inv x$ for $\omega$ a primitive third root of unity.
By Theorem \ref{thm.gJW}, $R^\grp{\eta}$ is again a GWA. In this case, $D^\grp{\eta}=\kk[h_1^2,h_2^2,h_1h_2]/(h_1^2h_2^2-(h_1h_2)^2)$ and $A=(h_1h_2)^3$.
\end{example}

We now consider two cases for fixed rings of quantum GWAs. The first corresponds to case (1) of Proposition \ref{prop.subgrp} and the second corresponds to case (2).

\subsection{The diagonal case}

Let $R=D(\sigma,a)$ be a quantum GWA with $a$ not symmetric. Suppose $\phi \in \Aut(R)$ acts diagonally on the algebra generating set $\{x,y,h\}$ of $R$ and $\ord(\phi)< \infty$.
By Proposition \ref{prop.subgrp}, there exists $\gamma,\mu \in \kk^\times$ with $n=\ord(\gamma)<\infty$ and $m=\ord(\mu)<\infty$ so that $\phi=\eta_{\gamma,\mu}$. That is, $\phi$ acts according to \eqref{eq.eta}. If we further assume that $n \mid i_0$, then the action is given simply by
\[ \phi(h)=\gamma h, \quad \phi(y) = \mu y, \quad \phi(x) = \mu\inv.\]
With these hypotheses the the fixed ring is easily computed using Thereom \ref{thm.gJW}.

\begin{corollary}
\label{cor.fixGWA1}
Keep the above notation and hypotheses. If $\gcd(m,n)=1$, then $R^\grp{\phi}$ again a quantum GWA.
\end{corollary}

In the setting of Corollary \ref{cor.fixGWA1} with $\eta=\eta_{\gamma,\mu}$,
$R^\grp{\phi}$ is generated by $Y=y^m$, $X=x^m$, and $H=h^n$ ($H^{\pm 1}=h^{\pm n}$ in the case $D=\lnt$). The defining polynomial is $A=\prod^{m-1}_{i=0} \sigma^{-i}(a) = \prod^{m-1}_{i=0} a(q^{-i}h)$ and $\sigma'(H)=q' H$, where $q'=q^{mn}$. 
Hence, if $D=\poly$, then $R^\grp{\phi}$ is the quantum GWA on $X,Y,H$ with relations
\[ XH=q'HX, \quad YH=(q')\inv HY, \quad
YX=A, \quad XY=\sigma'(A).\]
This extends readily to the case $D=\lnt$.

We conjecture that the converse to Corollary \ref{cor.fixGWA1} holds as well,
as the next example illustrates.

\begin{example}
Again keep the above notation and hypotheses but set $n=6$ and $m=4$.
Then $h^6$, $y^4$, and $x^4$ all belong to $R^\grp{\phi}$, but $y^2h^3$ and $x^2h^3$ are also fixed by $\phi$. Thus, $R^\grp{\phi}$ appears to be generated by more than three elements but every quantum GWA is minimally generated (as an algebra) by two or three elements. 
\end{example}

\subsection{The non-diagonal case}
Let $R=D(\sigma,a)$ be a quantum GWA with $a$ not symmetric and $q=-1$.
Suppose $\phi \in \Aut(R)$ does not act diagonally on the generating set $\{x,y,h\}$.
By Proposition \ref{prop.subgrp}, there exists $\gamma=\pm 1$ and $\mu \in D^\times$ such that 
$\phi=\Omega \circ \eta_{\gamma,\mu}$
For the convenience of the reader, we compute explicitly the image of each generator of $R$ under the map $\phi$:
\begin{align*}
\phi(h) &= \Omega(\eta_{\gamma,\mu}(h))
    = \Omega(\gamma h)
    = -\gamma h \\
\phi(y) &= \Omega(\eta_{\gamma,\mu}(y))
    = \Omega(y \mu)
    = x (-\gamma)^{\deg\mu} \mu 
    = \gamma^{\deg\mu} \mu x \\
\phi(x) &= \Omega(\eta_{\gamma,\mu}(x))
    = \Omega(\gamma^{i_0} \mu\inv x)
    = \gamma^{i_0} (-\gamma)^{-\deg\mu} \mu\inv y
    = \gamma^{i_0-\deg\mu} y\mu\inv.
\end{align*}
We now restrict ourselves to the case $\mu \in \kk^\times$, as well our previous hypothesis that $i_0 \mid \ord(\gamma)$. The above equations now reduce to
\[ \phi(h)=-\gamma h, \quad \phi(y)=\mu x, \quad \phi(x)=\mu\inv y.\]

The next lemma is clear.

\begin{lemma}
\label{lem.diag}
Keep the above notation and hypotheses. Set $X=(\mu x + y)$, $Y=(\mu x - y)$, and $H=h$.
Then $\{X,Y,H\}$ (resp. $\{X,Y,H^{\pm 1}\}$) is an algebra generating set for $R$ when $D=\poly$ (resp. $D=\lnt$) and $\phi$ acts diagonally on this generating set. That is, $\phi(X) = X$, $\phi(Y)=-Y$, and $\phi(H)=-\gamma H$.
\end{lemma}

The elements $X,Y,H$ from Lemma \ref{lem.diag} satisfy the commutation relations $0=HX + XH = HY + YH$. A computation shows
\begin{align}
\label{eq.xy1}
YX - XY
    = 2\mu (xy - yx) = 2\mu (a(-h)-a(h)) = A(H) \in \kk[H].
\end{align}
Finally, we have by a similar computation
\begin{align}
\label{eq.xy2}
X^2-Y^2=2\mu (a(h)+a(-h)) = B(H) \in \kk[H].
\end{align}
While $R$ does not necessarily have GWA structure under the generating set from Lemma \ref{lem.diag},
we can easily determine that $\{ X^iY^\varepsilon H^j : i,j \in \NN, \varepsilon \in \{0,1\} \}$ is a $\kk$-basis for $R$ when $D=\poly$. When $D=\lnt$, we allow $j \in \ZZ$.
Since $X$ is fixed, to determine a basis of $R^\grp{\phi}$ it suffices to determine those monomials $Y^iH^j$ that are fixed by $\phi$.
As $X$ and $Y^2$ are fixed by $\phi$ (for any choice of $\gamma$), then \eqref{eq.xy2} implies that $B(H)$ is also fixed by $\phi$. However, $\phi$ respects the $\NN$-grading on $\kk[h]$. 
Hence, $B(H) \in D^\grp{\phi}$.
Similarly, note that $\phi(A(H))=-A(H)$ where $A(H)$ is the polynomial appearing in \eqref{eq.xy1}.

We now consider the two cases determined by $\gamma=\pm 1$.

\begin{proposition}
\label{prop.omega}
Keep the above notation and assume $\gamma=-1$.
\begin{itemize}
\item If $D=\poly$, then $R^\grp{\phi} \iso \kk_{-1}[u,v] = \kk\langle u,v : uv+vu\rangle$.
\item If $D=\lnt$, then $R^\grp{\phi} \iso \kk_{-1}[u,v^{\pm 1}] = \kk\langle u,v^{\pm 1} : uv+vu\rangle$.
\end{itemize}
\end{proposition}
\begin{proof}
We prove this result for the case $D=\poly$. The case $D=\lnt$ is similar.

When $\gamma=-1$, it follows easily that the fixed ring is generated by $X$, $H$, and $Y^2$. 
However, by \eqref{eq.xy2}, $Y^2$ is a linear combination of $X^2$ and $H$. Hence, there is an isomorphism $R^\grp{\phi} \iso \kk_{-1}[X,H]$.
\end{proof}

If $R=D(\sigma,a)$ and $\phi=\Omega \circ \eta_{1,\mu} \in \Aut(R)$,
then the resulting fixed ring is less recognizable but still possible to compute.

\begin{proposition}
Keep the above notation and assume $\gamma=1$.
\begin{itemize}
\item If $D=\poly$, then $R^\grp{\phi}$ is generated by $X$, $Y^2$, $YH$, and $H^2$.
\item If $D=\lnt$, then $R^\grp{\phi}$ is generated by $X$, $Y^2$, $YH^{\pm 1}$, and $H^{\pm 2}$.
\end{itemize}
In both cases, the relations \eqref{eq.xy1} and \eqref{eq.xy2} hold, along with the relations
\[ 0 = Y(YH)+(YH)Y = Y^2X-XY^2 = X(YH) + (YH)X+A(H)H \quad\text{and}\quad H^2 \text{ central}.\]
\end{proposition}
\begin{proof}
It remains only to check the relations $Y^2X=XY^2$ and $X(YH) + (YH)X+A(H)H=0$.
We have
\[ Y^2X-XY^2=2\mu \left( \mu ( x^2y-yx^2) + (y^2x-xy^2) \right) = 0.\]
The last equality follows since $x^2y = x(xy)=x a(-h) = a(h) x = (yx)x=yx^2$ and similarly for the other term. 
From \eqref{eq.xy1} and \eqref{eq.xy2} we have
\[ X(YH) + (YH)X = (XY-YX)H = -A(H)H.\]
Clearly, $B(H) \in \kk[H^2]$.
Note that $A(H)$ necessarily consists of nonzero summands with odd degree, whence $A(H)H \in \kk[H^2]$.
\end{proof}

\section{Properties of the fixed rings}
\label{sec.props}

In this section we examine important ring theoretic properties of fixed rings of quantum GWAs including global dimension, the Calabi-Yau property, and simplicity. These properties rely by and large on the roots of the defining polynomial and this is where we begin our analysis.

We must take special care both when $q$ is a root of unity and when $0$ is a root of the defining polynomial $a$. However, we recall that by \cite[Theorem A]{SAV}, $\lnt(\sigma,a) \iso \lnt(\sigma,h^{-k}a)$ for any integer $k$. Hence, we need only consider $0$ as a root of $a$ in the case $D=\kk[h]$.

We first recall our notation from above. Let $R=D(\sigma,a)$ be a quantum GWA with $a$ not a unit. Write $a = \Sigma_{i \in I} a_i h^i$ where $a_i \in \kk$ and $I$ is a finite set consisting of those $i \in \ZZ$ such that $a_i \neq 0$. Recall from Remark \ref{rmk.indep} that the definition of $\eta_{\gamma,\mu}$ is independent of the choice of $i_0$.

\begin{lemma}
\label{lem.fixpoly}
Let $R=D(\sigma,a)$ be a quantum GWA and $\eta_{\gamma,\mu} \in \Aut(R)$. If $\ord(\gamma) \mid i_0$, then $a(h) \in D^\grp{\eta_{\gamma,\mu}}$.

\end{lemma}
\begin{proof}
Set $\ord(\gamma)=n$ and assume $n \mid i_0$. If $a(h)$ is a monomial, then the result is clear. Assume $a$ is not a monomial. Since $\gamma \in C_g$, then $n\mid g$. 
Since $\{i_0-j: a_j \neq 0\} \subset \{i-j: a_i a_j \neq 0\}$, then $n \mid (i_0-j)$ for every $j$ with $a_j\neq 0$. By assumption, we have $n \mid i_0$, so $n \mid j$. Hence, the conclusion holds.
\end{proof}

Throughout, we will work under the following hypotheses/notations which are inspired by the previous section. For $a(h) \in \kk[h]$ or $\lnt$, we denote by $N_a$ the number of (not necessarily distinct) roots of $a(h)$. Of course, in the case of $D=\kk[h]$, this is equivalent to the degree of $a(h)$.

\begin{hypothesis}
\label{hyp.gwa}
Let $R=D(\sigma,a)$ be a quantum GWA. Write 
\begin{align}
\label{eq.ah}
a(h) = (h-c_1)(h-c_2)(h-c_3) \dots (h-c_{N_a}).
\end{align}
Let $\gamma,\mu \in \kk^\times$ with $\ord(\gamma)= n$ and $\ord(\mu)= m$ satisfying $n,m < \infty$, $n \mid i_0$, and $\gcd(n,m)=1$. Set $\eta=\eta_{\gamma,\mu} \in \Aut(R)$.
By Lemma \ref{lem.fixpoly} and the discussion above, $a(h) \in \kk[h^n]$, so we write 
\begin{align}
\label{eq.bh}
b(h^n) = (h^n -d_1) (h^n -d_2)(h^n-d_3) \cdots (h^n-d_{N_b}),
\quad N_b=N_a/n.
\end{align}
Since we regard $b(h^n)$ as a polynomial in $h^n$, then we refer to $d_1,\hdots,d_{N_b}$ as the roots of $b$.
\end{hypothesis}

If $R$ and $\eta$ satisfy Hypothesis \ref{hyp.gwa}, then by Corollary \ref{cor.fixGWA1}, $R^{\grp{\eta}}$ again has quantum GWA structure. If we set $H=h^n$, then the defining polynomial for the fixed ring in the notation above is
\[ A(H)=\prod^{m-1}_{i=0} \sigma^{-i}(a(h))=\prod_{i=0}^{m-1} \sigma^{-i}(b(h^n)).\]

Two roots of $a(h)$, say $c_i$ and $c_j$, 
are said to be \emph {congruent} if $c_j=q^k c_i$ for some $k \in \ZZ$. To avoid trivialities, we will assume $i \neq j$ when discussing congruent roots, so that a root is not congruent to itself. However, as is standard, we consider multiple roots to be congruent.
We caution the reader that congruence is dependent on the parameter $q$, equivalently the map $\sigma$. That is, congruence in $R$ is different from congruence in $R^{\grp{\eta}}$. 

Below we study the relationship between congruent/multiple roots in $a(h)$ and in $A(H)$. The polynomial $b(h^n)$ will be used as an intermediate. Our analysis of the roots is covered by three lemmas. In Lemma \ref{lem.roots1} we consider the case of $q$ a nonroot of unity and $a(0) \neq 0$. Lemma \ref{lem.rootofunity} handles the root of unity case, also with $a(0) \neq 0$. Finally, the case in which $0$ is a root of $a(h)$ is explained in Lemma \ref{lem.root0}.

\begin{lemma}
\label{lem.roots1}
Let $R$ and $\eta$ satisfy Hypothesis \ref{hyp.gwa}. Suppose that $q$ is not a root of unity and $a(0) \neq 0$. 
\begin{enumerate}
\item The polynomial $a(h)$ has congruent (resp. multiple) roots if and only if $b(h^n)$ has congruent (resp. multiple) roots. 
\item The polynomial $A(H)$ has congruent roots if and only if $a(h)$ has congruent roots.
\item If $A(H)$ has multiple roots, then $a(h)$ has multiple or congruent roots.
\item If $a(h)$ has multiple roots, then $A(H)$ has multiple roots. Moreover, if $a(h)$ has congruent roots but no multiple roots, then $A(H)$ has multiple roots if and only if there exist roots $c_i$ and $c_j$ of $a(h)$ such that $c_i=q^k c_j$ for some $k$ with $0 < k \leq m-1$.
\end{enumerate}
\end{lemma}

\begin{proof} 
(1) First note that for all $j$, we have
\begin{align}
\sigma(h^n-d_j)=q^n h^n -d_j=q^n(h^n-q^{-n}d_j).
\end{align}
Then, $b(h^n)$ has congruent roots $d_i$ and $d_j$, $i \neq j$, if and only if $d_j = q^{ln}d_i$ for some $l \in \ZZ$. 

If $d_1$ is a root of $b(h^n)$ of multiplicity at least two, then every root of $(h^n-d_1)$ is a multiple root of $a(h)$.
Now assume $c_1$ is a root of $a(h)$ with multiplicity at least two. Then, $(h-c_1)^2 \mid a(h)$ so there exists a root $d_1$ of $b(h^n)$ such that $(h-c_1) \mid (h^n-d_1)$. But $c_1$ is not a root of $(h^n-d_1)/(h-c_1)$ so $d_1$ is a root of multiplicity at least two of $b(h^n)$.

If $b(h^n)$ has congruent (non-multiple) roots $d_1$ and $d_2$ such that $d_1 = q^{kn} d_2$ for some nonzero $k \in \ZZ$, then there exist roots $c_1$ and $c_2$ of $(h^n-d_1)$ and $(h^n-d_2)$, respectively, such that $c_1=q^k c_2$. Hence, $a(h)$ has congruent (non-multiple) roots $c_1$ and $c_2$. 
On the other hand, if $c_1$ and $c_2$ are congruent (non-multiple) roots of $a(h)$, then $c_1 = q^k c_2$ for some nonzero $k \in \ZZ$.  Then $b(h^n)$ has roots $d_1=c_1^n$ and $d_2=c_2^n$ with $d_1= q^{kn} d_2$. Thus $b(h^n)$ has congruent (non-multiple) roots $d_1$ and $d_2$.

(2) By \eqref{eq.bh},
\begin{align*}
A(H) 
    &= \prod_{i=0}^{m-1} \sigma^{-i}(b(h^n))
    = \prod_{i=0}^{m-1} \sigma^{-i}((h^n -d_1)(h^n -d_2) \cdots (h^n-d_{N_b})) \\
    &= \prod_{i=0}^{m-1} \sigma^{-i}(h^n -d_1)\sigma^{-i}(h^n -d_2) \cdots \sigma^{-i}(h^n-d_{N_b}) \\
    &= \prod_{i=0}^{m-1} q^{-in}(h^n-q^{in}d_1)q^{-in}(h^n -q^{in}d_2) \cdots q^{-in}(h^n-q^{in}d_{N_b}).
\end{align*}
For any $l \in \ZZ$ and $\alpha$ such that $0 \leq \alpha \leq m-1$,
\begin{align}
\label{eq.Acong}
(\sigma^m)^l(h^n-q^{\alpha n} d_i) = q^{mln}h^n - q^{\alpha n} d_i = q^{mln}(h^n-q^{(\alpha-ml)n}d_i).
\end{align}

Assume $A(H)$ has congruent roots, so there exists $l \in \ZZ$ and $0 \leq \alpha \leq \beta \leq m-1$ such that $q^{(\alpha-ml)n}d_i = q^{\beta n}d_j$ for some $i,j$.
Because $q$ is  not a root of unity, we may assume $i \neq j$.
Thus, $d_i = q^{(\beta-\alpha+ml)n} d_j$. We may assume above without loss that $\alpha=0$, so that $d_i = q^{(\beta+ml)n} d_j$. Thus, $b(h^n)$ has congruent roots $d_i$ and $d_j$. By (1), $a(h)$ has congruent roots. 

Now suppose that $a(h)$ has congruent roots, so $b(h^n)$ has congruent roots $d_1$ and $d_2$ by (1).
Then there exists $k\in \ZZ$ such that $d_1=q^{kn} d_2$.
Set $\beta \equiv k \mod m$ so that $0 \leq \beta \leq m-1$. By computations in the previous paragraph, $d_1$ and $q^{\beta n} d_2$ are congruent roots of $A(H)$.

(3) Assume $A(H)$ has multiple roots. Without loss of generality, and by (2), we can assume these are $d_1$ and $d_2$ with $d_1=q^{kn}d_2$ where $0 \leq k \leq m-1$. If $k=0$, then $b(h^n)$ has multiple roots so $a(h)$ has multiple roots by (1). If $k>0$, then $d_i$ and $d_j$ are congruent roots of $b(h^n)$, so $a(h)$ has congruent roots by (1).

(4) Suppose $a(h)$ has multiple roots. Then $b(h^n)$ has multiple roots by (1) and so $A(H)$ has multiple roots. Now suppose $a(h)$ has congruent roots $c_1$ and $c_2$ with $c_1=q^k c_2$, $0 < k \leq m-1$. Then $b(h^n)$ has congruent roots $d_1$ and $d_2$ with $d_1=q^{kn} d_2$ and these are multiple roots of $A(H)$. Conversely, if $A(H)$ has multiple roots but $a(h)$ has no multiple roots, then $a(h)$ has the claimed congruent roots by (3).
\end{proof}

\begin{remark}
The claim for multiple roots is not true without the root of unity hypothesis in Lemma \ref{lem.roots1} (1). For example, if $q$ is a primitive third root of unity and $a(h)=(h-1)(h-q)(h-q^2)$, then all of the roots of $a(h)$ are congruent. However, if $\ord(\gamma)=n=3$, then $b(h^n)=h^3-1$ and so $b(h^n)$ has no \emph{distinct} congruent roots. However, the remainder of (1) still holds without the hypothesis.

The root of unity hypothesis is also necessary in Lemma \ref{lem.roots1} (3). Suppose $a(h)=h-1$ with $n=1$ and $m=3$. Then, up to scalar multiple, $A(H)=(h-1)(h-q)(h-q^2)$. If $q=-1$, then $A(H)$ has multiple roots while $a(h)$ has neither multiple nor distinct congruent roots. 
\end{remark}

\begin{lemma}
\label{lem.rootofunity}
Let $R$ and $\eta$ satisfy Hypothesis \ref{hyp.gwa} with $q$ a root of unity, $q \neq 1$, and $a(0) \neq 0$. Then $A(H)$ has multiple roots if and only if one of the following hold:
\begin{enumerate}[label=(\roman*)]
    \item $a(h)$ has multiple roots,
    \item $a(h)$ has congruent roots $c_i$ and $c_j$, $i \neq j$, such that $c_i=q^k c_j$ with $0 < k \leq m-1$, or
    \item there exists $k$ such that $0 < k \leq m-1$ and $\ord(q)$ divides $nk$.
\end{enumerate}
\end{lemma}
\begin{proof}
That statements (i) and (ii) imply multiple roots in $A(H)$ follow similarly to Lemma \ref{lem.roots1} (4).
Assume (iii). Then $(h^n-d_1)$ and $(h^n-q^{kn} d_1)$ divide $A(H)$. It follows that $d_1$ is a multiple root of $A(H)$.

Suppose $A(H)$ has multiple roots. There are three cases to consider.
\begin{enumerate} 
\item Suppose $d_i = d_j$ for some $i \neq j$, then $b(h^n)$ has multiple roots and by the same argument as in Lemma \ref{lem.roots1} (1), $a(h)$ has multiple roots. 
\item Suppose $q^{n\alpha} d_i = q^{n\beta} d_j$ for some $d_i \neq d_j$ with $0 \leq \alpha \leq \beta \leq m-1$. Set $k=\beta-\alpha$. Since $d_i \neq d_j$, then $\ord(q) \nmid nk$ and so $b(h^n)$ has congruent roots. Thus, $a(h)$ has congruent roots satisfying (ii) by the same argument as in Lemma \ref{lem.roots1} (1).
\item Suppose $q^{n\alpha} d_i = q^{n\beta} d_i$ with $0 \leq \alpha < \beta \leq m-1$. Then $\ord(q) \mid nk$, where $k=\beta-\alpha$.\qedhere
\end{enumerate}
\end{proof}

We need to address the case where $0$ is a root of $a(h)$. In this case, it is possible for $a(h)$ to have multiple roots while $A(H)$ has no multiple roots. Here we do not need to differentiate between $q$ a root or nonroot of unity. 

\begin{lemma}
\label{lem.root0}
Let $R$ and $\eta$ satisfy Hypothesis \ref{hyp.gwa}. Let $k > 0$ be the multiplicity of $0$ as a root of $a(h)$ and write $a(h) =  h^k p(h)$. Then $A(H)$ has multiple roots if and only if one of the following hold:
\begin{enumerate}[label=(\roman*)]
    \item $m>1$,
    \item $k>n$,
    \item $q$ is not a root of unity and $p(h)$ satisfies the conditions of Lemma \ref{lem.roots1} (4), or
    \item $q$ is a root of unity and $p(h)$ satisfies one of the conditions of Lemma \ref{lem.rootofunity}.
\end{enumerate}
\end{lemma}
\begin{proof}
Recall that by \cite[Theorem A]{SAV} we may assume that $a(h)$ has positive degree. Since the multiplicity of $0$ as a root of $a(h)$ is $k>0$, then $a_k \neq 0$. Hence, we set $i_0=k$ so that $n \mid i_0$ by hypothesis. But $a(h)$ is a polynomial in $h^n$ by Lemma \ref{lem.fixpoly} and so it follows that $p(h)$ is also a polynomial in $h^n$.

If $k>n$, then $k/n$ is an integer greater than 1 and so $H^{k/n}$ is a factor of $A(H)$.
If $k=n$ but $m>1$, then $A(H) = \prod_{i=0}^{m-1} (q^{-in}H) \sigma^{-i}(p(h))$. In either case, $A(H)$ has a multiple root of $0$. Now assume $m=1$ and $k=n$, then $A(H) = H P(H)$ for some polynomial $P(H) \in \kk[H]$ with $P(0) \neq 0$ and so we can apply Lemma \ref{lem.roots1} (4).
\end{proof}

\subsection{Global dimension}

Recall that for any and ring $R$, the \emph{projective dimension} of a left $R$-module $M$, denoted $\pdim M$, is the shortest length of a projective resolution
\[ 0 \to P_n \to P_{n-1} \to \cdots \to P_0 \to M \to 0,\]
or $\infty$ if no such resolution exsists.
The \emph{global dimension} of $R$ is defined as 
\[ \gldim R = \sup\{ \pdim M : M \text{ a left $R$-module} \}.\]

In commutative algebra, global dimension is a measure of regularity of a ring. In algebraic geometry, it can be used to determine when an affine variety is nonsingular.
For quantum (and classical) GWAs, the global dimension is dictated by the roots of the defining polynomial and the associated automorphism. Hence, we will be able to apply our analysis above to determine the global dimension of fixed rings of quantum GWAs.

Let $R=D(\sigma,a)$ be a GWA over a Dedekind domain $D$. Suppose $a \neq 0$ and $a$ is not a unit.
Let $a=\bp_1^{n_1}\cdots \bp_s^{n_s}$ be the factorization of $a$ into the product of distinct maximal ideals of $D$. 
By \cite{bavgldim,Jkrull},
\[
\gldim R =
\begin{cases}
\infty & \text{if $n_i \geq 2$ for some $i$,} \\
2 & \text{$n_i=1$ for all $i$ and there exists a positive integer $k$ such that $\sigma^k(\bp_i)=\bp_j$} \\
& \text{ for some $i,j$ or $\sigma^k(\bq)=\bq$ for some maximal ideal $\bq$ of $D$,} \\
1 & \text{otherwise.}
\end{cases}
\]

\begin{lemma}
\label{lem.gldim}
Let $R=D(\sigma,a)$ be a quantum GWA.
\begin{itemize}
\item $R$ has infinite global dimension if and only if $a$ has multiple roots.
\item $R$ has global dimension 2 if and only if $a$ has no multiple roots and one of the following hold:
\begin{itemize}
    \item $D=\kk[h]$,
    \item $q$ is a root of unity, or
    \item $a$ has a pair of congruent roots.
\end{itemize}
\item $R$ has global dimension 1 otherwise.
\end{itemize}
\end{lemma}
\begin{proof}
The condition for infinite global dimension is just a restatement of the conditions above.
If $\sigma^k(\bp_i)=\bp_j$, then either $q$ is a root of unity or else $\bp_i$ and $\bp_j$ represent non-multiple congruent roots. Otherwise, $R$ has global dimension 2 if $\sigma^k(\bq)=\bq$ for some maximal ideal $\bq$ of $D$. This holds in the case $D=\kk[h]$ with $\bq=(h)$. In the case $D=\lnt$, this holds if and only if $q$ is a root of unity.
\end{proof}

The following example illustrates some of the behavior of global dimension in fixed rings of quantum GWAs.

\begin{example}
(1) Let $R=\lnt(\sigma,a)$ be a quantum GWA with 
$a=(h^2-1)(h^2-4)$ and $q=1/2$.
Then $a$ has congruent root pairs $\{1,2\}$ and $\{-1,-2\}$, so $\gldim R = 2$.
Let $\eta=\eta_{\gamma,\mu}$ with $n=2$ and $m=3$.
Then $H=h^2$ and $b(H)=(H-1)(H-4)$. Thus,
\begin{align*}
A(H)&= (H-1)(H-4)(4H-1)(4H-4)(16H-4)(16H-1) \\
    &= 4(H-1)^2 (H-4)(4H-1)(16H-1)(16H-4) \\
    &= 4^6 (H-1)^2(H-4)\left(H-\frac{1}{4}\right)^2\left(H-\frac{1}{16}\right).
\end{align*}
Hence, $A$ has multiple roots and so $\gldim R^\grp{\eta} = \infty$.

(2) Let $R=\poly(\sigma,a)$ be a quantum GWA with $a = h^2$.
Since $a$ has a multiple root of $0$ then $\gldim R = \infty$.
Let $\eta=\eta_{\gamma,\mu}$ with $n=2$ and $m=1$.
Then, $A(H)=H$ has no multiple root. Hence, $\gldim R^\grp{\eta} = 2$.
\end{example}

\begin{theorem}
\label{thm.gldim}
Suppose $R=D(\sigma,a)$ and $\eta$ satisfy Hypothesis \ref{hyp.gwa}. 
\begin{enumerate}
\item If $\gldim R = 1$, then $\gldim R^\grp{\eta} = 1$.
\item If $q$ is not a root of unity and $\gldim R = 2$, then $\gldim R^\grp{\eta} = \infty$ if and only if there exists roots $c_i$ and $c_j$ of $a(h)$ such that $c_i = q^kc_j$ for some $k$ with $0 < k \leq m-1$. Otherwise $\gldim R^\grp{\eta} = 2$.
\item If $q$ is a root of unity and $\gldim R = 2$, then $\gldim R^\grp{\eta} = \infty$ if and only if one of the conditions of Lemma \ref{lem.rootofunity} is satisfied.
\item If $\gldim R = \infty$, then $\gldim R^\grp{\eta} = 2$ if and only if $m=1$ and $0$ is a root of $a(h)$ with multiplicity $k=n$. Otherwise $\gldim R^\grp{\eta} = \infty$.
\end{enumerate}
\end{theorem}
\begin{proof}
(1) Suppose $\gldim R = 1$. By Lemma \ref{lem.gldim}, $D=\lnt$, $q$ is a nonroot of unity, and $a$ has neither multiple nor congruent roots. By Lemmas \ref{lem.roots1} (2) and (3), $A(H)$ also has neither multiple nor congruent roots. Hence, $\gldim R^\grp{\eta}=1$.

(2) By hypothesis, $a(h)$ has no multiple roots. Hence, if $A(H)$ has multiple roots, then $a(h)$ must have distinct congruent roots. The result now follows from (1) and from Lemma \ref{lem.roots1} (4).

(3) By hypothesis, $a(h)$ does not have multiple roots. Thus, $A(H)$ has multiple roots if and only if either the (ii) or (iii) of Lemma \ref{lem.rootofunity} is satisfied.

(4) If $a(h)$ has a multiple root other than $0$, then $A(H)$ has multiple roots by Lemma \ref{lem.roots1} (4). Now assume that $a(h)$ has only a multiple root of $0$. Then the result follows from Lemma \ref{lem.root0}.
\end{proof}

\begin{remark}
If we restrict to the case $D=\lnt$ and $q$ a nonroot of unity, then Theorem \ref{thm.gldim} is completely analogous to \cite[Corollary 2.12]{GW1}.
\end{remark}

The notion of a Calabi-Yau algebra was developed by Ginzburg as a way to port Calabi-Yau geometry to the language of noncommutative algebra \cite{ginz}. The more general notion of twisted Calabi-Yau algebras has drawn considerable interest of late, especially with its connection to noncommutative projective geometry \cite{RRZ,RR2}.

An algebra $A$ is \emph{homologically smooth} if it has a finitely generated projective resolution of finite length in $A^e=A \tensor_{\kk} A^\text{op}$. The algebra $A$ is \emph{twisted Calabi-Yau} of dimension $d$ if it is homologically smooth and there exists an invertible bimodule $U$ of $A$ such that $\Ext_{A^e}^i = \delta_{id} U$. If $U=A$, then $A$ is said to be \emph{Calabi-Yau}.

By work of Liu \cite[Theorem 4.5]{L}, a GWA $R=\kk[h](\sigma,a)$ is homologically smooth when $a$ has no multiple roots. By the above, this corresponds to the case that $R$ has finite global dimension. Moreover, in this case, $R$ is twisted Calabi-Yau with Nakayama automorphism $\nu$ given by
\[ \nu(x)=qx, \quad \nu(y)=q\inv y, \quad \nu(h)=h.\]
The next result now follows directly from Theorem \ref{thm.gldim}.

\begin{corollary}
\label{cor.CY}
Suppose that $R=\kk[h](\sigma,a)$ and $\eta$ satisfy Hypothesis \ref{hyp.gwa}. If $R$ is twisted Calabi-Yau, then $R^\grp{\eta}$ is twisted Calabi-Yau if and only if $R$ satisfies the hypotheses of (1) or (2) in Theorem \ref{thm.gldim}.
\end{corollary}

\subsection{Rigidity}
In \cite{Sm1}, Smith proved that a fixed ring of the first Weyl algebra by a nontrivial group can never be isomorphic to the first Weyl algebra. This result was extended by Alev and Polo to the $n$th Weyl algebra \cite{AP}. More recently, Tikaradze has shown that the first Weyl algebra is not the fixed ring of any domain \cite{TIK}. If $R=D(\sigma,a)$ is a classical GWA with $\deg_h(a)=2$ and $H$ a nontrivial finite cyclic group of filtered automorphisms, then $R^H \niso R$ by \cite[Corollary 2.11]{GW1}. We have an analogous result for quantum GWAs below.

\begin{proposition}
If $R$ and $\eta$ satisfy Hypothesis \ref{hyp.gwa} and $\eta$ is not the identity, then $R^\grp{\eta} \niso R$.
\end{proposition}
\begin{proof}
Suppose $R^\grp{\eta} \iso R$. Then by \cite[Theorem A]{SAV}, we must have $\deg_h(a)=\deg_H(A)$. 
Hence, $N_a=(N_a/n)m$ and so $1=m/n$.
By hypothesis, $\gcd(m,n)=1$. Thus, $m=n=1$. That is, $\eta$ is trivial.
\end{proof}

\subsection{Simplicity}

Let $R=D(\sigma,a)$ be a (not necessarily quantum) GWA. By \cite{B3,Jprim}, $R$ is simple if and only if 
\begin{enumerate}
    \item $D$ has no proper $\sigma$-stable ideal,
    \item no power of $\sigma$ is an inner automorphism of $D$, 
    \item $D=Da+D\sigma^n(a)$ for all $n \geq 1$, and
    \item $a$ is not a zero divisor in $D$.
\end{enumerate}

\begin{lemma}
Let $R=D(\sigma,a)$ be a quantum GWA.
Then $R$ is simple if and only if 
$D=\lnt$, $q$ is a nonroot of unity, and $a$ has no (non-multiple) congruent roots.
\end{lemma}
\begin{proof}
Since $(h)$ is a a $\sigma$-stable ideal of $D=\kk[h]$, then it follows that (1) holds if and only if $D=\lnt$ and $q$ is not a root of unity. Assume for the remainder that $D=\lnt$. Then (2) holds if and only if $q$ is not a root of unity (if $q^\ell=1$ then $\sigma^\ell=\id$). The ring $D$ has no zero divisors so it suffices to verify condition (3).
Note that (3) is equivalent to $\gcd(a,\sigma^n(a))=1$ for all $n \geq 1$, which holds if and only if $a$ has no congruent roots.
\end{proof}

\begin{proposition}
\label{prop.simple}
Suppose $R=\lnt(\sigma,a)$ and $\eta$ satisfy Hypothesis \ref{hyp.gwa}. Then $R^\grp{\eta}$ is simple if and only if $R$ is simple.
\end{proposition}
\begin{proof}
This follows directly from Lemma \ref{lem.roots1}.
\end{proof}

\begin{remark}
Let $R=\lnt(\sigma,h-1)$ be a quantum GWA. Then $R$ is isomorphic to a localization of the first quantum Weyl algebra. If $q$ is not a root of unity, then $R$ is simple by the above criteria. By Proposition \ref{prop.simple}, $R$ is not the fixed ring of any other quantum GWA by some $\eta$. It would be interesting to know whether a version of Tikaradze's theorem \cite{TIK} holds for $R$.
\end{remark}

\subsection*{Acknowledgements}
Phuong Ho was funded for this project through the Undergraduate Summer Scholars program at Miami University. The authors gratefully acknowledge this support. The authors thank Robert Won for helpful conversations and the anonymous referee for several suggested clarifications.

\bibliographystyle{amsplain}
\providecommand{\bysame}{\leavevmode\hbox to3em{\hrulefill}\thinspace}
\providecommand{\MR}{\relax\ifhmode\unskip\space\fi MR }
\providecommand{\MRhref}[2]{%
  \href{http://www.ams.org/mathscinet-getitem?mr=#1}{#2}
}
\providecommand{\href}[2]{#2}

\end{document}